\documentclass[a4paper]{amsart}

\usepackage{cite}
\usepackage{amsmath}
\usepackage{amssymb}
\usepackage{amsthm}   
\usepackage{empheq}   
\usepackage{graphicx}
\usepackage{color}
\usepackage{appendix}
\usepackage{hyperref}

\newtheorem{theorem}{Theorem}[section]
\newtheorem{lemma}[theorem]{Lemma}

\newtheorem{corollary}[theorem]{Corollary}

\theoremstyle{definition}

\newtheorem{example}[theorem]{Example}

\theoremstyle{remark}
\newtheorem{remark}[theorem]{Remark}

\numberwithin{equation}{section}

\def\EE{{\mathcal{E}}}
\def\FF{{\mathcal{F}}}
\def\tt{{\mathtt{t}}}
\def\ss{{\mathtt{s}}}

\def\cc{{\mathfrak{c}}}
\def\mm{{\mathfrak{m}}}

\begin{document}

\title[Orthogonal complement]{The orthogonal complements of $H^1(\mathbb{R})$ in its regular Dirichlet extensions}

\author{Yuncong Shen}
\address{School of Mathematical Sciences, Fudan University, Shanghai 200433, China.}
\email{yuncongshen13@fudan.edu.cn}

\author{ Liping Li*}
\address{Institute of Applied Mathematics, Academy of Mathematics and Systems Science, Chinese Academy of Sciences, Beijing 100190, China.}
\email{liping\_li@amss.ac.cn}
\thanks{*Partially supported by a joint grant (No. 2015LH0043) of China Postdoctoral Science Foundation and Chinese Academy of Science, and China Postdoctoral Science Foundation (No. 2016M590145).}

\author{Jiangang Ying**}
\address{School of Mathematical Sciences, Fudan University, Shanghai 200433, China.}
\email{jgying@fudan.edu.cn}
\thanks{** Partially supported by NSFC No. 11271240.}

\subjclass[2010]{Primary 31C25, 60J55; Secondary 60J60}

\date{\today}


\keywords{Regular Dirichlet extensions, regular Dirichlet subspaces, Dirichlet forms, orthogonal complements, darning processes}

\begin{abstract}
Consider the regular Dirichlet extension $(\EE,\FF)$ for one-dimensional Brownian motion, that $H^1(\mathbb{R})$ is a subspace of $\FF$ and $\EE(f,g)=\frac12\mathbf{D}(f,g)$ for $f,g\in H^1(\mathbb{R})$. Both $H^1(\mathbb{R})$ and $\FF$ are Hilbert spaces under $\EE_\alpha$ and hence there is $\alpha$-orthogonal compliment $\mathcal{G}_\alpha$. We give the explicit expression for functions in $\mathcal{G}_\alpha$ which then can be described by another two spaces. On the two spaces, there is a natural Dirichlet form in the wide sense and by the darning method, their regular representations are given.
\end{abstract}

\maketitle

\tableofcontents

\section{Introduction}
The notion of regular Dirichlet subspace was first raised for one-dimensional Brownian motion in \cite{FFY05}, and then studied in \cite{LY14, LY15} for the general cases and their structures. Precisely, let $E$ be a locally compact separable metric space and $m$ a Radon measure on $E$ fully supported on $E$. If two regular Dirichlet forms $(\mathcal{E}^1,\mathcal{F}^1)$ and $(\mathcal{E}^2,\mathcal{F}^2)$ on $L^2(E,m)$ satisfy
$$
\mathcal{F}^1\subset\mathcal{F}^2,\quad\mathcal{E}^2(u,v)=\mathcal{E}^1(u,v),\quad\forall u,v\in\mathcal{F}^1,
$$
then $(\mathcal{E}^1,\mathcal{F}^1)$ is called a regular Dirichlet subspace of $(\mathcal{E}^2,\mathcal{F}^2)$, and $(\mathcal{E}^2,\mathcal{F}^2)$ is called a regular Dirichlet extension of $(\mathcal{E}^1,\mathcal{F}^1)$. We refer the terminologies about the Dirichlet forms to \cite{CF12} and \cite{FOT11}.


It is well known that the associated Dirichlet form of one-dimensional Brownian motion on $L^2(\mathbb{R})$ is $(\frac12\mathbf{D},H^1(\mathbb{R}))$, where $H^1(\mathbb{R})$ is 1-Sobolev space and 
\[
	\mathbf{D}(u,v)=\int_\mathbb{R}u'(x)v'(x)dx,\quad u,v\in H^1(\mathbb{R}).
\]
The Lebesgue measure on $\mathbb{R}$ will also be denoted by $m$ throughout this paper. 
A characterization for regular Dirichlet subspaces of $(\frac12\mathbf{D},H^1(\mathbb{R}))$ was given in \cite{FFY05}. 
It turns out that every regular Dirichlet subspace of $(\frac12\mathbf{D},H^1(\mathbb{R}))$ is irreducible and its associated diffusion process may be characterized by a scale function $\ss$ on $\mathbb{R}$, which is absolutely continuous and $\ss'=0$ or $1$ a.e. (Cf. \cite{FFY05} and \cite[Proposition~2.1]{LS16-1}). Later in \cite{LY16}, the authors gave a complete characterization for regular Dirichlet extensions of $(\frac12\mathbf{D},H^1(\mathbb{R}))$. Although every regular Dirichlet extension of $(\frac12\mathbf{D},H^1(\mathbb{R}))$ is strongly local (Cf. \cite[Theorem~1]{LY15}) and thus it always corresponds to a diffusion process with no killing inside, the challenge is that unlike the subspace, this diffusion process is not necessarily irreducible. As in \cite{LY16}, let $a<b$ and $I=\langle a,b\rangle$ be one of $(a,b)$, $(a,b]$, $[a,b)$, and $[a,b]$. Denote by $\mathbf{T}(I)$ all strictly increasing and continuous functions $\mathtt t(x)$ on $I$ such that
$$
dx\ll d\mathtt t,\quad\frac{dx}{d\mathtt t}=0\text{ or }1,\quad d\tt\text{-a.e.}
$$
where `$\ll$' means `absolutely continuous'. We can extend $\mathtt t(x)$ to $[a,b]$ because $\mathtt t(x)$ is monotonic. Define
\begin{equation}\label{EQ2TIT}
\mathbf{T}_\infty(I):=\{t\in \mathbf{T}(I):\;\mathtt t(a)=-\infty\text{ iff }a\notin I,\text{ and }\mathtt t(b)=\infty\text{ iff }b\notin I\},
\end{equation}
where `iff' stands for `if and only if'.  Note that $\mathbf{T}_\infty(I)$ is not empty (Cf. \cite[Remark~3.1]{LY16}). 
One of the main results, Theorem~3.3, of \cite{LY16} showed that every regular Dirichlet extension $(\EE,\FF)$ of $(\frac12\mathbf{D},H^1(\mathbb{R}))$ is expressed as follows: there exist a series of at most countable disjoint intervals $\{I_n=\langle a_n,b_n\rangle:n\ge1\}$ satisfying $m((\bigcup_{n\ge1}I_n)^c)=0$ and a series of scale functions $\{\mathtt t_n\in \mathbf{T}_\infty(I_n):n\ge1\}$ such that
\begin{equation}\label{EQ1FFL}
\begin{aligned}
&\mathcal{F} = \left\{f\in L^2(\mathbb{R}): f|_{I_n}\in\mathcal{F}^n, \sum_{n\ge1}\mathcal{E}^n(f|_{I_n},f|_{I_n})<\infty \right\},\\
&\mathcal{E}(f,g) = \sum_{n\ge1}\mathcal{E}^n(f|_{I_n},g|_{I_n}),\quad f,g\in\mathcal{F},
\end{aligned}\end{equation}
where $(\mathcal{E}^n,\mathcal{F}^n)$ is defined by
\begin{align*}
&\mathcal{F}^n=\left\{f\in L^2(I_n): f\ll \mathtt t_n, \int_{I_n}\left(\frac{df}{d\mathtt t_n}\right)^2d\mathtt t_n<\infty\right\},\\
&\mathcal{E}^n(f,g)=\frac12\int_{I_n}\frac{df}{d\mathtt t_n}\frac{dg}{d\mathtt t_n}d\mathtt t_n,\quad f,g\in\mathcal{F}^n.
\end{align*}
Note that $f\ll \tt_n$ means that there exists an absolutely continuous function $g$ such that $f=g\circ \tt_n$ and $df/d\tt_n :=g'\circ \tt_n$. In other words, the extension $(\EE,\FF)$ can be decomposed into at most countable irreducible parts. Its associated diffusion process on each interval $I_n$ is an irreducible diffusion process with the scale function $\tt_n$ and speed measure $m|_{I_n}$. Moreover, the condition $m((\bigcup_{n\ge1}I_n)^c)=0$ indicates $(\bigcup_{n\ge1}I_n)^c$ is an $\EE$-polar set relative to $(\EE,\FF)$ (Cf. \cite[Remark~3.4(4)]{LY16}). Several examples of extensions are given in \cite[\S3.3]{LY16}. 

Since $(\frac12\mathbf{D},H^1(\mathbb{R}))$ is a regular Dirichet subspace of $(\EE,\FF)$ given by \eqref{EQ1FFL}, it follows that $H^1(\mathbb{R})$ is a closed subspace of $\FF$ with respect to the inner product $\EE_\alpha$ for any constant $\alpha>0$. Naturally, we can define the $\alpha$-orthogonal complement of $H^1(\mathbb{R})$ in $\FF$ by 
\begin{equation}\label{EQ1GAF}
	\mathcal{G}_\alpha:=\left\{f\in \FF: \EE_\alpha(f,g)=0,\forall g\in H^1(\mathbb{R}) \right\}. 
\end{equation}
We also write $\FF=H^1(\mathbb{R})\oplus_\alpha \mathcal{G}_\alpha$ or $\mathcal{G}_\alpha= \FF\ominus_\alpha H^1(\mathbb{R})$. 
Replacing $H^1(\mathbb{R})$ and $\FF$ by their extended Dirichlet spaces, the special case $\alpha=0$ was investigated in \cite[\S4.1]{LY16}. The main purpose of this paper is to explore the $\alpha$-orthogonal complement $\mathcal{G}_\alpha$ for any fixed constant $\alpha>0$. Recall that the same topic on the regular Dirichlet subspace of $(\frac12\mathbf{D},H^1(\mathbb{R}))$ was already considered in \cite{LS16}. There are also some discussions with similar form in Example 1.2.2 of \cite{FOT11}.

The structure of this paper is as follows. First we give the expression of $\mathcal{G}_\alpha$ as in \cite{LS16}, and then derive another explicit expression consists two functions $\cc_+$ and $\cc_-$. We show that $\cc_\pm$ form a Hilbert space $\mathfrak C_\pm$ to which there is a linear bijection from $\mathcal{G}_\alpha$. Furthermore, the natural quadratic form $\EE_\pm$ on $\mathfrak C_\pm$ is a Dirichlet form in the wide sense. Finally, we use the darning method to derive the regular representation of $(\EE_\pm,\mathfrak C_\pm)$ on each interval $I_n$.

\subsection*{Notations}
For each $n\ge1$, set
$$
U_n:=\left\{x\in I_n:\frac{dx}{d\mathtt t_n}=1\right\}.
$$
Note that $U_n$ is defined in the sense of $d\tt_n$-a.e. Set $W_n:= I_n\setminus U_n$. 
Clearly, 
\[
\begin{aligned}
&m(W_n)=0, \\ 
&\frac{dx}{d\tt_n}=1_{U_n}, \quad d\mathtt t_n\text{-a.e.},
\end{aligned}\]
which will be used frequently. Hence some fact holds `a.e. on $I_n$' is equivalent to it holds `$d\tt_n$-a.e. on $U_n$'. Set further
\[
	U:=\bigcup_{n\geq 1}U_n, \quad W:=U^c,\quad I:=\bigcup_{n\geq 1}I_n.
\]
Following \cite{LY16}, we write $I_n=\langle a_n,b_n\rangle$ where $a_n$ (resp. $b_n$) may or may not be in $I_n$. 
Given a function $f$ on $\mathbb{R}$ and a set $A\subset \mathbb{R}$, the restriction of $f$ to $A$ is denoted by $f|_A$. The restriction of a measure $m$ to $A$ is denoted by $m|_A$. In addition, if $\mathbf{S}$ is some symbol, then $\mathbf{S}_\pm$ means $\mathbf{S}_+$ or $\mathbf{S}_-$, i.e., the sentence or equation in which it is located holds for both $\mathbf{S}_+$ and $\mathbf{S}_-$.


\section{Characterizations of orthogonal complements}\label{SEC2}
Let $(\mathcal{E},\mathcal{F})$ be \eqref{EQ1FFL} a regular Dirichlet extension of $(\frac12 \mathbf{D},H^1(\mathbb{R}))$ on $L^2(\mathbb{R})$, $\{I_n=\langle a_n,b_n\rangle: n\ge1\}$ and $\{\mathtt t_n: n\ge1\}$ be the associated intervals and scale functions. Fix a constant $\alpha>0$ and $\mathcal{G}_\alpha$  given by \eqref{EQ1GAF} is the $\alpha$-orthogonal complement of $H^1(\mathbb{R})$ in $\FF$.

Denote by $H_e^1(\mathbb{R})$ and $\mathcal{F}_e$ the extended Dirichlet spaces of $H^1(\mathbb{R})$ and $\FF$ respectively. In \cite{LY16} the authors investigated the orthogonal complement of $H_e^1(\mathbb{R})$ in $\mathcal{F}_e$:
$$
\mathcal{G}:=\{f\in\mathcal{F}_e:\mathcal{E}(f,g)=0,\forall g\in H_e^1(\mathbb{R})\}.
$$
\cite[Theorem 4.2]{LY16} tells us that
\begin{equation}\label{POC}
\mathcal{G}=\left\{f\in\mathcal{F}_e: \frac{df|_{I_n}}{d\mathtt t_n}=0,d\mathtt t_n \text{-a.e. on }U_n \text{ for any }n\ge1 \right\},
\end{equation}
and $f\in\mathcal{F}_e$ can be decomposed uniquely up to a constant as
$$
f=f_1+f_2,\quad f_1\in H_e^1(\mathbb{R}),f_2\in\mathcal{G}.
$$
We use the similar method to derive the first expression of $\mathcal{G}_\alpha$ in the following theorem.

\begin{theorem}\label{T1}
	For any constant $\alpha>0$,
\begin{equation}\label{OC}
	\mathcal{G}_\alpha=\bigg\{f\in\mathcal{F}:\; \frac{df|_{I_n}}{d\mathtt t_n}(y) - \frac{df|_{I_m}}{d\tt_m}(x)=2\alpha\int_x^yf(z)dz, 
	\text{ a.e. } x\in I_m, y\in I_n,\; \forall m,n\ge1\bigg\}. 
\end{equation}
\end{theorem}
\begin{proof}
	Fix two functions $f\in\mathcal{F}$ and $g\in C_c^\infty(\mathbb{R})$ and suppose that $\text{supp}[g]$ is in some interval $(a,b)$. Then
	\begin{align*}
	\int_\mathbb{R}f(x)g(x)dx&=-\int_a^bf(x)dx\int_x^bg'(y)dy\\
	&=-\int_a^bg'(y)dy\int_a^yf(x)dx.
	\end{align*}
	Letting $f_n=f\cdot1_{I_n}$, we have
	\begin{align*}
	\mathcal{E}_\alpha(f,g)&=\frac12\sum_{n\ge1}\int_a^b\frac{df_n}{d\mathtt t_n}(y) g'(y)dy-\alpha\int_a^bg'(y)dy\int_a^yf(x)dx\\
	&=\int_a^bg'(y)dy\left[\frac12\sum_{n\ge1}\frac{df_n}{d\mathtt t_n}(y)- \alpha\int_a^yf(x)dx\right].
	\end{align*}
	Now assume that $f$ is in the right side of (\ref{OC}). Then for a.e. $x,y\in(\bigcup_{n\ge1}I_n)\cap (a,b)$ with $x\in I_m,y\in I_n$, we have
	\begin{align*}
	\frac12\sum_{n\ge1}\frac{df_n}{d\mathtt t_n}(x)- \alpha\int_a^xf(z)dz &=\frac12\frac{df_m}{dt_m}(x)- \alpha\int_a^xf(z)dz\\
	&=\frac12\frac{df_n}{d\mathtt t_n}(y)- \alpha\int_a^yf(z)dz\\
	&=\frac12\sum_{n\ge1}\frac{df_n}{d\mathtt t_n}(y)- \alpha\int_a^yf(z)dz \equiv C,
	\end{align*}
	where $C$ is a constant. Since $m((a,b)\setminus\bigcup_{n\ge1}I_n)=0$, it follows that
	$$
	\mathcal{E}_\alpha(f,g)=C\cdot\int_a^bg'(y)dy=0.
	$$
	Hence $f\in\mathcal{G}_\alpha$ because $C_c^\infty(\mathbb{R})$ is $\mathcal{E}_1$-dense in $H^1(\mathbb{R})$.
	
	On the contrary, assume $f\in\mathcal{G}_\alpha$. Then
	$$
	\mathcal{E}_\alpha(f,g)=0,\quad \forall g\in C_c^\infty(\mathbb{R}).
	$$
	Suppose $\text{supp}[g]\subset(a,b)$ as above. Let
	$$
	h(y):=\frac12\sum_{n\ge1}\frac{df_n}{d\mathtt t_n}(y)-\alpha\int_a^yf(x)dx,\quad y\in(a,b).
	$$
	Then
	$$
	\mathcal{E}_\alpha=\int_a^bg'(y)h(y)dy=0,\quad \forall g\in C_c^\infty(\mathbb{R}).
	$$
	Thus $h(x)$ is a constant for a.e. $x\in(a,b)$. It follows that
	$$
	0=2h(y)-2h(x)=\frac{df|_{I_n}}{d\mathtt t_n}(y) - \frac{df|_{I_m}}{dt_m}(x) - 2\alpha\int_x^yf(z)dz
	$$
	for a.e. $x\in I_m\cap(a,b),y\in I_n\cap(a,b)$. Since $(a,b)$ can be taken arbitrarily large, $f$ is in the right side of (\ref{OC}), which completes the proof.
\end{proof}
\begin{remark}
	In the case of $\alpha=0$, formally we have
\begin{equation}\label{EQ2GFF}
	\mathcal{G}_0=\left\{f\in\mathcal{F}:\frac{df|_{I_n}}{d\mathtt t_n}=c,\text{ a.e. on }I_n,\forall n\ge1 \right\}
\end{equation}
for some constant $c$. 
However, since $f\in\mathcal{G}_0\subset \FF$, it follows that
	$$
	\infty>\mathcal{E}(f,f)=\sum_{n\ge1}\mathcal{E}^n(f|_{I_n},f|_{I_n}) \ge \frac{c^2}{2}\sum_{n\ge1}\int_{I_n}1_{U_n}(x)\cdot d\mathtt t_n(x) = \frac{c^2}{2} m(\mathbb{R}).
	$$
Therefore $c=0$ and
	$$
	\mathcal{G}_0=\left\{f\in\mathcal{F}:\frac{df|_{I_n}}{d\mathtt t_n}=0,\text{ a.e. on }I_n,\forall n\ge1 \right\},
	$$
	which is the same as (\ref{POC}) if we replace $\FF$ by $\FF_\mathrm{e}$.
\end{remark}

Now we turn to formulate another explicit expression of $\mathcal{G}_\alpha$. Keep in mind that the function $f\in \FF$ is continuous on each interval $I_n$, but not necessarily continuous on $\mathbb{R}$. In fact, the restriction $f|_{I_n}$ of $f$ to $I_n$ is absolutely continuous with respect to $\tt_n$ by \eqref{EQ1FFL}, thus it is also continuous on $I_n$. For the second assertion, we refer an example to \cite[Example~4.4]{LY16}. Furthermore, $f$ does not have point-wise definition on $\left(\bigcup_{n\geq 1} I_n\right)^c$ since $m(\left(\bigcup_{n\geq 1} I_n\right)^c)=0$.

We first assume that a $d\tt_n$-version of $U_n$ is open for any $n\geq 1$, which will be cancelled in Theorem~\ref{T2}. This is a very natural assumption since the typical example of $W_n=I_n\setminus U_n$ is a Cantor-type set as in \cite{FFY05} and \cite[\S4.2]{LY16}. As \cite[(4.7)]{LY16}, write $U_n$ as a union of disjoint open intervals:
\begin{equation}\label{EQ2UNM}
	U_n=\bigcup_{m\geq 1} (a^n_m,b^n_m). 
\end{equation}
Recall that $U=\bigcup_{n\ge1}U_n$ and $W=U^c$. Fix $f\in \mathcal{G}_\alpha$ and positive $n,m\in \mathbb{N}$. Since $d\tt_n=dx$ on $(a^n_m,b^n_m)$, it follows from Theorem \ref{T1} that
$$
f'(y)-f'(x)=2\alpha\int_x^yf(z)dz,\quad x, y \in (a^n_m,b^n_m),
$$
and we attain the harmonic equation similar to \cite[\S3]{LS16}:
\begin{equation}\label{FF}
\frac12f''(x)=\alpha f(x),\quad x\in (a^n_m,b^n_m). 
\end{equation}
Clearly, $f$ is twice differentiable in $(a^n_m,b^n_m)$. Thus the solution of (\ref{FF}) is 
\[
	f(x)=\frac{\mathfrak{c}^{n,m}_+}{2}\cdot e^{\sqrt{2\alpha}x}+\frac{\mathfrak{c}^{n,m}_-}{2}\cdot e^{-\sqrt{2\alpha}x},\quad x\in (a^n_m,b^n_m)
\]
for two constants $\mathfrak{c}^{n,m}_+, \mathfrak{c}^{n,m}_-$. Therefore we have
\begin{equation}\label{EQ2FXC}
f(x)=\frac{\mathfrak{c}_+(x)}{2}e^{\sqrt{2\alpha}x}+\frac{\mathfrak{c}_-(x)}{2}e^{-\sqrt{2\alpha}x},\quad x\in U,
\end{equation}
with two functions $\mathfrak{c}_+,\mathfrak{c}_-$ satisfying $\mathfrak{c}_\pm(x)\equiv \mathfrak{c}^{n,m}_\pm$ for $x\in (a^n_m,b^n_m)$. This implies
\begin{equation}\label{EQ2CTT}
	\frac{d\mathfrak{c}_\pm}{d\tt_n}=0,\quad d\tt_n\text{-a.e. on }U_n, \; n\geq 1. 
\end{equation}
Roughly speaking, $\mathfrak{c}_\pm$ is a Cantor-type function on each interval $I_n$ as we noted above that $W_n$ is a Cantor-type set. 
Since $f$ is continuous on $I=\bigcup_{n\geq 1}I_n$ and $f|_{I_n}\ll \tt_n$, we expect that $\mathfrak{c}_\pm$ can be extended continuously to a function on $I$, which is still denoted by $\mathfrak{c}_\pm$, and $\mathfrak{c}_\pm|_{I_n}\ll \tt_n$. In other words, \eqref{EQ2FXC} holds for any $x\in I$ with two functions $\mathfrak{c}_+$ and $\mathfrak{c}_-$ on $I$ satisfying
\begin{itemize}
\item[\textbf{(C1)}] $\mathfrak{c}_\pm|_{I_n}\ll \tt_n$ for any $n\geq 1$ and \eqref{EQ2CTT} holds. 
\end{itemize}
Furthermore, if \textbf{(C1)} holds, some simple computations deduce that for any $n\geq 1$,
\begin{align}
&\frac{df|_{I_n}}{d\tt_n}(x)=\frac{\sqrt{2\alpha}}{2}\left(\mathfrak{c}_+(x)e^{\sqrt{2\alpha}x}-\mathfrak{c}_-(x)e^{-\sqrt{2\alpha}x} \right),\quad d\tt_n\text{-a.e. } x\in U_n, \label{EQ2FIT} \\
&\frac{df|_{I_n}}{d\tt_n}(x)=\frac{1}{2}\left(\frac{d\mathfrak{c}_+|_{I_n}}{d\tt_n}(x)e^{\sqrt{2\alpha}x}+\frac{d\mathfrak{c}_-|_{I_n}}{d\tt_n}(x)e^{-\sqrt{2\alpha}x}\right),\quad d\tt_n\text{-a.e. } x\in W_n. \label{EQ2DFT}
\end{align}
Denote the right side of \eqref{EQ2FIT} by $u(x)$ 
for $x\in U$. It follows from Theorem~\ref{T1} that 
\[
u(y)-u(x)=2\alpha \int_x^y f(z)dz, \quad \text{a.e. }x,y \in U.  
\]
Since $f\in L^2(\mathbb{R})\subset L^1_\mathrm{loc}(\mathbb{R})$, we can obtain 
\begin{itemize}
\item[\textbf{(C2)}] $u(x):=\frac{\sqrt{2\alpha}}{2}\left(\mathfrak{c}_+(x)e^{\sqrt{2\alpha}x}-\mathfrak{c}_-(x)e^{-\sqrt{2\alpha}x} \right),\; x\in U$ can be extended to an absolutely continuous function on $\mathbb{R}$. 
\end{itemize}
We need to point out that although $\mathfrak{c}_\pm$ is a constant on $(a^m_n,b^m_n)$ and hence absolutely continuous on $U$, it is usually not absolutely continuous on $W=\bigcup_{n\geq 1}W_n$. 

The following theorem indicates that \textbf{(C1)} and \textbf{(C2)} are not only necessary but also sufficient, and the assumption that $U_n$ is open is not essential. 

\begin{theorem}\label{T2}
	Suppose that $f\in \FF$. Then $f\in\mathcal{G}_\alpha$ if and only if
	\begin{equation}\label{FCC}
	f(x)=\frac{\mathfrak{c}_+(x)}{2}e^{\sqrt{2\alpha}x} + \frac{\mathfrak{c}_-(x)}{2}e^{-\sqrt{2\alpha}x},\quad x\in I,
	\end{equation}
	where $(\mathfrak{c}_+,\mathfrak{c}_-)$ is a pair of functions defined on $I$ satisfying \textbf{(C1)} and \textbf{(C2)}. Furthermore, $\mathfrak{c}_\pm$ in \eqref{FCC} is uniquely determined by $f$.
\end{theorem}
{\begin{proof}
	Suppose $f\in\mathcal{G}_\alpha$. It follows from \eqref{OC} that there exists an absolutely continuous function $u$ on $\mathbb{R}$ such that $u'(x)=2\alpha f(x)$ on $I$ and
	\begin{equation}\label{UFT}
	u(x)=\frac{df|_{I_n}}{d\mathtt{t}_n}(x),\quad d\mathtt{t}_n\text{-a.e. }x\in U_n.
	\end{equation}
	Define $\mathfrak{c}_+,\mathfrak{c}_-$ by
	\begin{equation}\label{EQ2CXC}
	\begin{aligned}
	&\mathfrak{c}_+(x):=e^{-\sqrt{2\alpha}x}\left( f(x)+\frac{1}{\sqrt{2\alpha}}u(x) \right), \\
	&\mathfrak{c}_-(x):=e^{\sqrt{2\alpha}x}\left( f(x)-\frac{1}{\sqrt{2\alpha}}u(x) \right)
	\end{aligned}\end{equation}
	for $x\in I$. Then we have \eqref{FCC}, and \textbf{(C2)} holds. Clearly $\mathfrak{c}_{\pm}|_{I_n}\ll\tt_n$ and the density on each $I_n$ is
	\[
	\frac{d\mathfrak{c}_\pm}{d\mathtt{t}_n}=e^{\mp\sqrt{2\alpha}x}\left( \frac{df}{d\mathtt{t}_n}- u(x)\frac{dx}{d\mathtt{t}_n}\right).
	\]
	Hence \eqref{EQ2CTT} follows from \eqref{UFT}.

	On the contrary, suppose that $f$ is defined by \eqref{FCC} with $(\cc_+,\cc_-)$ satisfying \textbf{(C1)} and \textbf{(C2)}. By \eqref{EQ2FIT} and \textbf{(C2)}, we have
	\[
	\frac{df|_{I_n}}{d\tt_n}(x)=u(x),\quad d\tt_n\text{-a.e. } x\in U_n,
	\]
	where $u$ is absolutely continuous on $\mathbb{R}$. Hence for $d\mathtt{t}_m$-a.e. $x\in U_m$ and $d\mathtt{t}_n$-a.e. $y\in U_n$,
	\begin{align*}
	\frac{df|_{I_n}}{d\mathtt{t}_n}(y)-\frac{df_{I_m}}{d\mathtt{t}_m}(x) &= \int_x^yu'(z)dz \\
	&=\sum_{n\ge1}\int_{(x,y)\cap U_n} \frac{du}{d\mathtt{t}_n}d\mathtt{t}_n\\
	&=\sum_{n\ge1}\int_{(x,y)\cap U_n} \alpha \left( \mathfrak{c}_+(z)e^{\sqrt{2\alpha}z} + \mathfrak{c}_-(z)e^{-\sqrt{2\alpha}z}\right) \frac{dz}{d\mathtt{t}_n}d\mathtt{t}_n.\\
	&=2\alpha\sum_{n\ge1}\int_{(x,y)\cap U_n} f(z)dz \\
	&=2\alpha\int_x^y f(z)dz.
	\end{align*}
The uniqueness of $\mathfrak{c}_\pm$ follows from \eqref{EQ2CXC}. 
That completes the proof.
\end{proof}

\begin{remark}\label{RM24}
\begin{itemize}
\item[(1)] It follows from the proof that if $(\cc_+$, $\cc_-)$ satisfies \textbf{(C1)} and \textbf{(C2)}, then
\[
	\frac{d\mathfrak{c}_+|_{I_n}}{d\tt_n}(x)e^{\sqrt{2\alpha}x}=\frac{d\mathfrak{c}_-|_{I_n}}{d\tt_n}(x)e^{-\sqrt{2\alpha}x},\quad d\tt_n\text{-a.e. } x\in W_n, n\geq 1.
\]
In other words, the two terms in the right side of \eqref{EQ2DFT} are equal. 
\item[(2)] When $f$ has the form of \eqref{FCC}, $f\in \FF$ can be read as the $L^2$-integrability of \eqref{FCC}, \eqref{EQ2FIT}, and \eqref{EQ2DFT}, which are equivalent respectively to
\[
	\int_\mathbb{R}\mathfrak{c}_+(x)^2e^{2\sqrt{2\alpha}x}dx<\infty,\quad \int_\mathbb{R}\mathfrak{c}_-(x)^2e^{-2\sqrt{2\alpha}x}dx<\infty,
\]
and
\[
\sum_{n\ge1}\int_{W_n}\left( \frac{d\mathfrak{c}_+}{d\tt_n}e^{\sqrt{2\alpha}x} \right)^2 d\tt_n = \sum_{n\ge1}\int_{W_n}\left( \frac{d\mathfrak{c}_-}{d\tt_n}e^{-\sqrt{2\alpha}x} \right)^2 d\tt_n<\infty, 
\]
\item[(3)] When $\alpha=0$, clearly $\mathfrak{c}_+=\mathfrak{c}_-$ and \textbf{(C2)} is trivial ($u\equiv 0$). Hence \eqref{FCC} is actually the same as the condition in \eqref{EQ2GFF}. 
\end{itemize}
\end{remark}
}

The following is a very simple example of $(\cc_+$, $\cc_-)$ satisfying \textbf{(C1)} and \textbf{(C2)}. 

\begin{example}\label{EXA25}
Fix an integer $n$ and $-\infty=d_0<d_1<\cdots<d_n<d_{n+1}=\infty$. Let $I_i:=(d_{i-1}, d_i)$ for $1\leq i\leq n+1$ be the intervals associated with the extension $(\EE,\FF)$. Take two sets of constants $\{\mathfrak{c}_+^{(i)}:i=1,\dots,n+1\}$ and $\{\mathfrak{c}_+^{(i)}:i=1,\dots,n+1\}$ such that
	\[
	\mathfrak{c}_+^{(i)}e^{\sqrt{2\alpha}d_i}- \mathfrak{c}_-^{(i)}e^{-\sqrt{2\alpha}d_i}= \mathfrak{c}_+^{(i+1)}e^{\sqrt{2\alpha}d_i}- \mathfrak{c}_-^{(i+1)}e^{-\sqrt{2\alpha}d_i},\quad i=1,\dots,n
	\]
and	$\mathfrak{c}_-^{(1)}=\mathfrak{c}_+^{(n+1)}=0$. Define 
\[
	\mathfrak{c}_\pm(x):=\cc^{(i)}_\pm,\quad x\in I_i, \; 1\leq i\leq n+1. 
\]
Clearly, $(\cc_+,\cc_-)$ satisfies \textbf{(C1)} and \textbf{(C2)}. Furthermore, it is easy to check that $f\in \FF$ and hence $f\in \mathcal{G}_\alpha$.
In particular, for $a=d_i$ ($1\leq i\leq n$), 
$$
f(x):=\begin{cases}
e^{\sqrt{2\alpha}(x-a)}\quad & x<a,\\
-e^{\sqrt{2\alpha}(a-x)}\quad & x>a.
\end{cases}
$$
Then $f$ is in $\mathcal{G}_\alpha$.  
\end{example}

\section{Darning processes}

The purpose of this section is to derive an induced `darning process' from the class of functions $\mathfrak{c}_\pm$, which are discussed in \S\ref{SEC2}. We first recall the darning process of $\mathcal{G}$ in \cite[\S4.2]{LY16}. Note that $\mathcal{G}$ is given by \eqref{POC}, which is a pseudo orthogonal complement of $H^1_\mathrm{e}(\mathbb{R})$ in $\FF_\mathrm{e}$ with respect to $\EE$. Clearly $\mathcal{G}$ is not a `real' Dirichlet space since it is not dense in $L^2(\mathbb{R})$. However, in \cite[Lemma~4.6]{LY16} the authors assert that $\mathcal{G}$ with the form $\EE$ is a Dirichlet form in the wide sense on $L^2(\mathbb{R})$, which means that it satisfies all the conditions of a Dirichlet form except for the denseness in $L^2(\mathbb{R})$. Note that $(\mathbb{R}, m, \mathcal{G}, \EE)$ is called a `D-space' in Fukushima's terminology.  It is proved in \cite{F71} by Fukushima that every D-space (say $(\mathbb{R}, m, \mathcal{G}, \EE)$) has a regular representation in the sense that there exist a regular Dirichlet form $(\EE',\mathcal{G}')$ on another Hilbert space $L^2(E',m')$ and an algebraic isomorphism $\Phi$ from $\mathcal{G}\cap L^\infty(\mathbb{R})$ to $\mathcal{G}'\cap L^\infty(E')$ such that 
\begin{equation}\label{EQ3UUU}
\|u\|_\infty=\|\Phi u\|_\infty,\quad \EE(u,u)=\EE'(\Phi u, \Phi u),\quad (u,u)_m=(\Phi u, \Phi u)_{m'}. 
\end{equation}
Following \cite[\S3.2]{LY14}, the authors introduce the `darning method' to find the regular representation of $(\EE, \mathcal{G})$ in \cite[\S4.2]{LY16}. Roughly speaking, under the basic assumption `$U_n$ is open', any function in $\mathcal{G}$ is a constant on the interval $(a_m^n, b_m^n)$ of \eqref{EQ2UNM}. The darning method is regarding this interval as a whole part and collapsing it into a new point. After doing this on every interval of \eqref{EQ2UNM}, they attain a regular Dirichlet form on the new state space and it is actually the desired regular representation. We refer more details to \cite[\S3.2]{LY14} and \cite[\S4.2]{LY16}. However, the following example indicates that $(\EE, \mathcal{G}_\alpha)$ usually does not satisfy the Markovian property and hence is not a Dirichlet form in the wide sense. Therefore, the darning method on $\mathcal{G}$ in \cite{LY16} is not available for $\mathcal{G}_\alpha$. 


\begin{example}
Let us consider a special case that $(\EE,\FF)$ is irreducible, in other words, $I_1=I=\mathbb{R}$ and the scale function $\tt$ in $\mathbf{T}_\infty(\mathbb{R})$ is as follows: $\tt(x)=x+c(x)$, where $c$ is the standard Cantor function on $[0,1]$ with $c(x):=0$ for $x<0$ and $c(x):=1$ for $x>1$.  Clearly, 
\[
	U=\left\{x\in \mathbb{R}: \frac{dx}{d\tt}=1\right\}=(-\infty, 0)\bigcup \left( [0,1]\setminus K\right) \bigcup (1,\infty),
\]
where $K$ is the standard Cantor set in $[0,1]$. It follows from Theorem~\ref{T2} that the restriction of any function $f\in \mathcal{G}_\alpha$ to $(-\infty, 0)$ is expressed as
\[
	f(x)=k_+e^{\sqrt{2\alpha}x}+k_-e^{-\sqrt{2\alpha}x},\quad x<0,
\]
with two constants $k_+$ and $k_-$. Remark~\ref{RM24}~(2) yields $k_-=0$. Without loss of generality, we can take a function $f\in \mathcal{G}_\alpha$ such that 
\[
	f(x)=k_+e^{\sqrt{2\alpha}x},\quad x<0
\]
with $k_+>1$. Consider the normal contraction $\phi(x)=1\wedge x \vee 0$. Clearly $\phi\circ f\notin \mathcal{G}_\alpha$ and thus $\mathcal{G}_\alpha$ does not satisfy the Markovian property. 
\end{example}

The condition \eqref{EQ2CTT} in \S\ref{SEC2} inspires us to transfer the focus to the functions $\mathfrak{c}_\pm$. As a rough observation, if we put $\alpha=0$, then $\mathfrak{c}_+=\mathfrak{c}_-=f$ and it is exactly a function in $\mathcal{G}$, which satisfies the Markovian property as proved in \cite[Lemma~4.6]{LY16}. In what follows, we turn to characterize the class of all possible functions $\mathfrak{c}_\pm$. 

For convenience, write
\[
	h_\pm(x):=e^{\pm\sqrt{2\alpha}x}. 
\] 
From Theorem~\ref{T2} we know that $\mathfrak{c}_\pm$ is uniquely determined by $f$. Thus we may denote 
\[
	\mathfrak{c}_\pm:=\Gamma_\pm f,\quad f\in \mathcal{G}_\alpha
\]	
and induce linear mapping $\Gamma_\pm$ on $\mathcal{G}_\alpha$. Let
\begin{equation}\label{EQ3MXH}
\mathfrak{m}_\pm(dx):= h_\pm(x)^2dx.
\end{equation}
Define
\begin{equation}\label{EQ3CCR}
\begin{aligned}
	\mathfrak{C}_\pm:=\bigg\{\cc\in L^2(\mathbb{R};\mm_\pm): \ &\cc|_{I_n}\ll \tt_n,\ \frac{d\cc|_{I_n}}{d\tt_n}=0,\ d\tt_n\text{-a.e. on }U_n, \\
	&\sum_{n\geq 1}\int_{I_n}\left(\frac{d\cc}{d\tt_n}\cdot h_\pm\right)^2d\tt_n<\infty \bigg\},
\end{aligned}\end{equation}
\begin{equation}\label{EQ3ECC}
	\mathcal{E}_\pm(\cc_1,\cc_2):=\frac{1}{2} \sum_{n\geq 1}\int_{I_n}\frac{d\cc_1}{d\tt_n}\frac{d\cc_2}{d\tt_n}h_\pm^2d\tt_n,\quad\cc_1,\cc_2\in \mathfrak{C}_\pm,
\end{equation}
and
\[
	\EE_{\pm,1}(\cc_1,\cc_2):=\EE_\pm(\cc_1,\cc_2)+(\cc_1,\cc_2)_{\mathfrak{m}_\pm}. 
\]
Note that we regard $\cc_1,\cc_2\in\mathfrak C_\pm$ as equal if and only if $\cc_1(x)=\cc_2(x)$ for $x\in I$. Then we have the following result to describe the mapping $\Gamma_\pm$.

\begin{theorem}\label{TBIJ}
The mapping
\[
	\Gamma_\pm: \mathcal{G}_\alpha \rightarrow \mathfrak{C}_\pm, f\mapsto \cc_\pm
\]
induced by \eqref{FCC} is a linear bijection and for any $f\in \mathcal{G}_\alpha$, 
\begin{equation}\label{EQ3EGF}
	\EE_\pm(\Gamma_\pm f, \Gamma_\pm f)= \frac{1}{2}\int_W  d\mu_{\langle f\rangle}, 
\end{equation}
where $\mu_{\langle f\rangle}$ is the energy measure of $f$ relative to $(\EE,\FF)$. Furthermore, there exists a constant $K_\alpha>1$ such that for any $f\in \mathcal{G}_\alpha$,
\begin{equation}\label{EQ3KEK}
K_\alpha^{-1}\EE_{\pm, 1}(\Gamma_\pm f, \Gamma_\pm f)\leq \EE_1(f,f)\leq K_\alpha\EE_{\pm, 1}(\Gamma_\pm f, \Gamma_\pm f).
\end{equation}
\end{theorem}
To proceed, we prepare a lemma.
\begin{lemma}
If $g\in L^2(\mathbb{R})$, then for any $\lambda>0$,
\begin{equation}\label{EQ3GG}
\begin{aligned}
&\int_\mathbb{R}\left[e^{-\lambda x}\int_{-\infty}^xg(z)e^{\lambda z}dz\right]^2dx\le \lambda^{-2}\|g\|_{L^2}^2,\\
&\int_\mathbb{R}\left[e^{\lambda x}\int_x^\infty g(z)e^{-\lambda z}dz\right]^2dx\le \lambda^{-2}\|g\|_{L^2}^2.
\end{aligned}
\end{equation}
\end{lemma}
\begin{proof}
Let $G(x)=e^{-\lambda x}\int_{-\infty}^xg(z)e^{\lambda z}dz$. By Cauchy-Schwarz inequality,
$$
G(x)^2\le e^{-2\lambda x}\int_{-\infty}^xg(z)^2dz\int_{-\infty}^xe^{2\lambda z}dz= \frac{1}{2\lambda}\int_{-\infty}^xg(z)^2dz\to 0
$$
as $x\to-\infty$. Since $G'=-\lambda G+g$, we have
\begin{align*}
\|g\|_{L^2}^2 &\ge\int_{-\infty}^ag(x)^2dx\\
&=\int_{-\infty}^a\left( {G'}^2+2\lambda G'G+\lambda^2G^2\right)dx\\
&=\int_{-\infty}^a\left( {G'}^2+\lambda^2G^2\right)dx+ \lambda G(a)^2\\
&\ge \lambda^2 \int_{-\infty}^aG^2dx.
\end{align*}
Letting $a\to\infty$ gives the first inequality and the second follows similarly.
\end{proof}
\begin{proof}[Proof of Theorem~\ref{TBIJ}]
The linearity is obvious. To prove $\Gamma_\pm$ is a bijection, it is exactly to show that for any $\cc_\pm\in\mathfrak C_\pm$, there is a unique $\cc_\mp\in\mathfrak C_\mp$ defined on $I$ satisfying \textbf{(C2)}.

Without loss of generality, we fix a $\cc_+\in\mathfrak C_+$, and assume that $\cc_-\in\mathfrak C_-$ satisfying \textbf{(C2)}. Then $v(x):=\cc_+(x)h_+(x)^2-\cc_-(x)$ has an absolutely continuous extension on $\mathbb{R}$ which we still denote by $v$. Hence for $x,y\in I$,
\begin{align*}
\cc_-(y)-\cc_-(x) &=\cc_+(y)h_+(y)^2- \cc_+(x)h_+(x)^2-\int_x^yv'(z)dz\\
&=\cc_+(y)h_+(y)^2- \cc_+(x)h_+(x)^2-\sum_{n\ge1}\int_{(x,y)\cap U_n}\frac{dv}{d\mathtt{t}_n}d\mathtt{t}_n\\
&=\cc_+(y)h_+(y)^2- \cc_+(x)h_+(x)^2- 2\sqrt{2\alpha}\int_x^y\cc_+(z)h_+(z)^2 dz.
\end{align*}
It follows that
$$
\cc_-(x)-\cc_+(x)h_+(x)^2+2\sqrt{2\alpha}\int_{-\infty}^x\cc_+(x)h_+(z)^2dz \equiv C
$$
for some constant $C$. By \eqref{EQ3GG}, we can see that $Ch_-(x)$ is in $L^2(\mathbb{R})$. Therefore $C=0$ and
\begin{equation}\label{EQ3CCC}
\cc_-(x)=\cc_+(x)h_+(x)^2 - 2\sqrt{2\alpha}\int_{-\infty}^x\cc_+(z)h_+(z)^2dz,\quad x\in I.
\end{equation}
Clearly $\cc_-$ defined by \eqref{EQ3CCC} is unique in $\mathfrak C_-$ and satisfies \textbf{(C2)}. Hence $\Gamma_\pm$ is a bijection. By  \eqref{EQ2FIT}, \eqref{EQ2DFT} and \eqref{FCC}, we have \eqref{EQ3EGF} and
\[
	\frac12 \int_Ud\mu_{\langle f\rangle} + (f,f)_m= \frac{\alpha}{2}\int_\mathbb{R} ((\Gamma_+f)h_+ - (\Gamma_-f)h_-)^2dx + \frac14 \int_\mathbb{R}((\Gamma_+f)h_+ + (\Gamma_-f)h_-)^2dx
\]
for $f\in\mathcal{G}_\alpha$. It follows from \eqref{EQ3GG} and \eqref{EQ3CCC} that
\[
	\int_\mathbb{R}((\Gamma_+f)h_+ - (\Gamma_-f)h_-)^2dx \le 4(\Gamma_\pm f,\Gamma_\pm f)_{\mathfrak m_\pm}.
\]
A little computation gives
\[
	\left(\alpha\wedge\frac12\right)\mathcal{E}_{\pm,1}(\Gamma_\pm f,\Gamma_\pm f) \le \mathcal{E}_1(f,f) \le (2\alpha+4)\mathcal{E}_{\pm,1}(\Gamma_\pm f,\Gamma_\pm f).
\]
That completes the proof.
\end{proof}

As a result, we assert that $(\EE_\pm, \mathfrak{C}_\pm)$ is a Dirichlet form on $L^2(\mathbb{R},\mathfrak{m}_\pm)$ in the wide sense. When (formally) $\alpha=0$, clearly $h_\pm\equiv 1$, $\mathfrak{C}_\pm=\mathcal{G}_0=\mathcal{G}\cap L^2(\mathbb{R})$ and $\EE_\pm=\EE$. This is nothing but what has been considered in \cite{LY16}.

\begin{corollary}\label{COR33}
Let $\mathfrak{C}_\pm, \EE_\pm$ and $\mathfrak{m}_\pm$ be given by \eqref{EQ3CCR}, \eqref{EQ3ECC} and \eqref{EQ3MXH} respectively. Then $(\EE_\pm, \mathfrak{C}_\pm)$ is a Dirichlet form on $L^2(\mathbb{R}, \mathfrak{m}_\pm)$ in the wide sense. 
\end{corollary}
\begin{proof}
It is easy to check that $(\EE_\pm, \mathfrak{C}_\pm)$ is a symmetric form with Markovian property. The closedness follows from \eqref{EQ3KEK} and the closedness of $(\EE,\mathcal{G}_\alpha)$.
\end{proof}

\begin{corollary}\label{C-FC}
For any $\alpha>0$, $f\in\mathcal{G}_\alpha$ if and only if
\[
f(x)=\cc_+(x)h_+(x)-\sqrt{2\alpha}h_-(x)\int_{-\infty}^x c_+(z)h_+(z)^2dz,\quad\cc_+\in\mathfrak C_+,
\]
or equivalently
\[
f(x)=\cc_-(x)h_-(x)-\sqrt{2\alpha}h_+(x)\int_x^\infty c_-(z)h_-(z)^2dz,\quad\cc_-\in\mathfrak C_-.
\]
\end{corollary}

Now we shall apply the darning method introduced in the beginning of this section. Note that the key to the darning method is an induced transform by $\tt_n$. Since $\tt_n(I_n)$ is unbounded when $I_n$ is not closed (Cf. \S\ref{EQ2TIT}), it is better to consider the restriction of $(\EE_\pm, \mathfrak{C}_\pm)$ to each invariant interval $I_n$ than $(\EE_\pm, \mathfrak{C}_\pm)$ itself. Fix an integer $n\geq 1$. Let
\[\begin{aligned}
	\mathfrak{C}^n_\pm&:=\{\cc|_{I_n}: \cc\in \mathfrak{C}_\pm\} \\
	 &=\bigg\{\cc \in L^2(I_n;\mm_\pm|_{I_n}): \cc\ll \tt_n, \frac{d\cc}{d\tt_n}=0, d\tt_n\text{-a.e. on }U_n,\, \int_{I_n}\left(\frac{d\cc}{d\tt_n}h_\pm\right)^2d\tt_n<\infty \bigg\},
	\end{aligned}\]
and for any $\cc \in \mathfrak{C}^n_\pm$, 
\[
\EE^n_\pm(\cc, \cc):= \frac{1}{2}\int_{I_n}\left(\frac{d\cc}{d\tt_n}\right)^2h_\pm^2d\tt_n.
\]
Further set $\mathfrak{m}^n_\pm:=\mathfrak{m}_\pm|_{I_n}$. Similar to Corollary~\ref{COR33} we can deduce that $(\EE^n_\pm ,\mathfrak{C}^n_\pm)$ is a Dirichlet form on $L^2(I_n, \mathfrak{m}^n_\pm)$ in the wide sense.

Recall that $I_n=\langle a_n, b_n\rangle$. Hereafter, we enforce the basic assumptions in \cite[\S4.2]{LY16}:
\begin{itemize}
\item[\bf{(H1)}] $U_n$ has (and is taken as) a $d\tt_n$-a.e. open version, and $U_n$ is written as \eqref{EQ2UNM}. 
\item[\bf{(H2)}] For any $x\in W_n\cap (a_n,b_n)$ and $\epsilon>0$, $d\tt_n((x-\epsilon, x+\epsilon)\cap W_n)>0$. 
\item[\bf{(H3)}] $d\tt_n(W_n)>0$. 
\end{itemize} 
Note that \textbf{(H2)} is not essential as noted in \cite[\S4.2]{LY16} and we refer the explanation for the structures of $U_n$ and $W_n$ under these assumptions to \cite[Remark~4.5]{LY16}. 

We shall describe the darning method to find the regular representation of $(\EE^n_\pm, \mathfrak{C}^n_\pm)$. Since the presence of weight function $h_\pm$, the case here is more complicated than that in \cite{LY16}. Let
\[
	l_n:=\inf\{x: x\in W_n\},\quad r_n:=\sup\{x: x\in W_n\}. 
\] 
Define the following transform on $I_n$ which collapses each open interval $(a_m^n, b_m^n)$ in \eqref{EQ2UNM} into a point: 
\begin{equation}\label{EQ3JNX}
	j_n(x):=\int_{e_n}^x1_{W_n}(y)d\tt_n(y),\quad x\in I_n,
\end{equation}
 where $e_n$ is a fixed point in $(a_n,b_n)$. Further set
 \[
	l_n^{*}:=j_n(l_n)\geq -\infty,\quad r_n^{*}:= j_n(r_n)\leq \infty.  
 \]
 For the right endpoints $r_n$ and $r_n^{*}$, there are the following possible situations:
\begin{itemize}
\item[\bf{(R1)}] $b_n<\infty, b_n\notin I_n$: since $U_n$ is open and $\tt_n(b_n)=\infty$, it follows that $r_n=b_n$ and $r^{*}_n=\infty$. 
\item[\bf{(R2)}] $b_n<\infty, b_n\in I_n$: we have $r_n=b_n$ and $r^{*}_n<\infty$. 
\item[\bf{(R3)}] $b_n=\infty$: 
\begin{itemize}
\item[\bf{(R3i)}] $r_n=\infty, r_n^{*}<\infty$. This is possible such as the following example. Without loss of generality, assume $I_n=[0,\infty)$. Take a standard Cantor function $c$ in $[0,1]$. Define a function $C(x)$ on $[0, \infty)$ as follows: $C(x):=c(x)$ for $x\in [0,1]$ and for any integer $k\geq 1$ and $x\in [k,k+1]$,
\[
	C(x):= 1+\frac{1}{2^2}+\cdots +\frac{1}{k^2} +\frac{1}{(k+1)^2}\cdot c(x-k). 
\]
Further set $\tt_n(x):=x+C(x)$. Then we have $r_n=\infty$ and 
\[
	r^{*}_n=1+\frac{1}{2^2}+\cdots <\infty. 
\]
\item[\bf{(R3ii)}] $r_n=\infty, r_n^{*}=\infty$. In the above example, replacing the definition of $C$ by
\[
	C(x):=k+c(x-k),\quad x\in [k,k+1],
\]
we have an example for this case. 
\item[\bf{(R3iii)}] $r_n<\infty, r^{*}_n<\infty$. As in Example~\ref{EXA25}, $r_1=r^{*}_1=1$. 
\end{itemize}
\end{itemize}
The corresponding cases for the left endpoints $l_n$ and $l^{*}_n$ are denoted by \textbf{(L1)}, \textbf{(L2)}, \textbf{(L3)}, \textbf{(L3i)}, \textbf{(L3ii)} and \textbf{(L3iii)}. Define the following intervals
\[
	J_+^{n*}:=\langle l^*_n, r^*_n \rangle, 
\]
where $r^*_n\in J_+^{n*}$ for and only for the case \textbf{(R2)} and $l^*_n\in J^{n*}_+$ for and only for the cases \textbf{(L2)}, \textbf{(L3i)} and \textbf{(L3iii)}, and
\[
J_-^{n*}:=\langle l^*_n, r^*_n \rangle, 
\]
where $r^*_n\in J_-^{n*}$ for and only for the cases \textbf{(R2)}, \textbf{(R3i)}, \textbf{(R3iii)} and $l^*_n\in J^{n*}_-$ for and only for the case \textbf{(L2)}. Note that in the cases \textbf{(R1)} and \textbf{(R3ii)} (resp. \textbf{(L1)} and \textbf{(L3ii)}), $r^*_n=\infty$ (resp. $l^*_n=-\infty$). 


The darning transform $j_n$ induces an image measure of $\mm^n_\pm$ on $J^{n*}_\pm$
\[
	\mm^{n*}_\pm:= \mm^n_\pm\circ j_n^{-1}
\]
by the convention $\mm^{n*}_\pm(J^{n*}_\pm\setminus j_n(I_n))=0$.  One may easily check that $\mm^{n*}_\pm$ is a Radon measure on $J^{n*}_\pm$ with full support. Note that $\mm^{n*}_\pm$ is not absolutely continuous with respect to the Lebesgue measure. Denote its Lebesgue decomposition by 
\[
	\mm^{n*}_\pm=h^*_\pm(x)^2dx+\kappa_\pm,
\]
where $h^*_\pm(x)^2dx$ is the absolutely continuous part. 

Fix a function $\cc\in \mathfrak{C}_\pm$. Since $\cc\ll \tt_n$ and $\cc$ is a constant on $(a^n_m,b^n_m)$, the darning transform induces an absolutely continuous function $\cc^*$ on $j_n(I_n)$ such that
\[
	\cc(x)=\cc^*\circ j_n(x),\quad x\in I_n. 
\]
We can recompute $\EE^n_\pm(\cc,\cc)$ as follows:
\[
	\EE^n_\pm(\cc,\cc)=\frac{1}{2} \int_{l_n}^{r_n} \left(\frac{d\cc^*\circ j_n}{d\tt_n}\right)^2h^2_\pm d\tt_n=\frac{1}{2}\int_{l^*_n}^{r^*_n} \left(\frac{d\cc^*}{dx}\right)^2 \left(h^*_\pm\right)^2dx. 
\]
Denote the last term above by $\EE^{n*}_\pm(\cc^*,\cc^*)$. Further define
\[
	\mathfrak{C}^{n*}_\pm:=\left\{\cc^*\in L^2(J^{n*}_\pm, \mm^{n*}_\pm): \exists\,\cc\in\mathfrak C_\pm^n\text{ s.t. }\cc^*(x)=\cc\circ j_n^{-1}(x)\text{ for }x\in J^{n*}_\pm\cap j_n(I_n) \right\}. 
\]
It can be seen that $j_n^{-1}$ is an isomorphism from $(\mathfrak{C}^{n}_\pm,\EE_\pm^n)$ to $(\mathfrak{C}^{n*}_\pm,\EE_\pm^{n*})$. Now we write the definition explicitly according to different cases.
\begin{center}
\begin{tabular}{|l|c|c|c|}
	\hline
	&$J_+^{n*}$	&$J_-^{n*}$	&$j_n(I_n)$	\\ \hline
	\textbf{(R1)} & $\langle l_n^*,\infty)$ & $\langle l_n^*,\infty)$ & $\langle l_n^*,\infty)$ \\ \hline
	\textbf{(R2)} & $\langle l_n^*,r_n^*]$ & $\langle l_n^*,r_n^*]$ & $\langle l_n^*,r_n^*]$ \\ \hline
	\textbf{(R3i)} & $\langle l_n^*,r_n^*)$	& $\langle l_n^*,r_n^*]$ & $\langle l_n^*,r_n^*)$ \\ \hline
	\textbf{(R3ii)} & $\langle l_n^*,\infty)$ & $\langle l_n^*,\infty)$ & $\langle l_n^*, \infty)$ \\ \hline
	\textbf{(R3iii)} & $\langle l_n^*,r_n^*)$ & $\langle l_n^*,r_n^*]$ & $\langle l_n^*,r_n^*]$ \\ \hline
\end{tabular}
\end{center}
Excluding the cases \textbf{(L3i)} \textbf{(R3iii)} for $(\mathfrak{C}^{n*}_+, \EE_+^{n*})$, and \textbf{(R3i)} \textbf{(L3iii)} for $(\mathfrak{C}^n_-, \EE_-)$, we have $J^{n*}_\pm= j_n(I_n)$, hence
\[
\mathfrak{C}^{n*}_\pm=\left\{\cc^*\in L^2(J^{n*}_\pm, \mm^{n*}_\pm): \cc^*\text{ is absolutely continuous},\,\EE^{n*}_\pm(\cc^*,\cc^*)<\infty \right\}.
\]
For the case $(\mathfrak{C}^{n*}_+, \EE_+^{n*})$ with \textbf{(L3i)}, $l_n^*\in J_+^{n*}$ but $l_n^*\notin j_n(I_n)$, hence
\[
\mathfrak{C}^{n*}_+=\left\{\cc^*\in L^2(J^{n*}_+, \mm^{n*}_+): \cc^*\text{ is absolutely continuous on }(l_n^*,r_n^*\rangle,\, \EE^{n*}_+(\cc^*,\cc^*)<\infty \right\},
\]
whether $r_n^*\in(l_n^*,r_n^*\rangle$ depends on the case of \textbf{(R)}. For the case $(\mathfrak{C}^{n*}_+, \EE_+^{n*})$ with \textbf{(R3iii)}, $r_n^*\notin J_\pm^{n*}$ but $r_n^*\in j_n(I_n)$. Note that any $\cc\in \mathfrak{C}^n_+$ is a constant on $[r_n,\infty)$ whereas $\mm^n_+((r_n,\infty))=\infty$. Thus $\cc=0$ on $[r_n,\infty)$. Therefore
\begin{align*}
\mathfrak{C}^{n*}_+=\big\{&\cc^*\in L^2(J^{n*}_+, \mm^{n*}_+): \cc^*\text{ is absolutely continuous on }\langle l_n^*,r_n^*),\\
&\lim_{x\uparrow r_n^*}\cc^*(x)=0,\,\EE^{n*}_+(\cc^*,\cc^*)<\infty \big\},
\end{align*}
whether $l_n^*\in\langle l_n^*,r_n^*)$ depends on the case of \textbf{(L)}. The case $(\mathfrak{C}^{n*}_-, \EE_-^{n*})$ with \textbf{(R3i)} and \textbf{(L3iii)} is similar. The main result of this section is the following. Keep in mind that the adjustment to the endpoints of $J_\pm^{n*}$ relative to $j_n(I_n)$ is to ensure the regularity.

\begin{theorem}
The quadratic form $(\EE^{n*}_\pm, \mathfrak{C}^{n*}_\pm)$ is a regular Dirichlet form on $L^2(J^{n*}_\pm, \mm^{n*}_\pm)$. Furthermore, $(J^{n*}_\pm, \mm^{n*}_\pm,  \mathfrak{C}^{n*}_\pm, \EE^{n*}_\pm)$ is a regular representation of $(I_n, \mm^n_\pm, \mathfrak{C}^{n}_\pm, \EE^{n}_\pm)$. 
\end{theorem}

\begin{proof}
We prove the regularity for \textbf{(L2)} and all cases of \textbf{(R)}, then the other cases follow immediately. For simplicity, we assume $l^*_n=0$.

In case (\textbf{R2)}, $\mathfrak{C}^{n*}_\pm\subset C([0, r^*_n])$; and in \textbf{(R3iii)}, $\mathfrak{C}^{n*}_-\subset C([0, r^*_n])$. The regularity is clear.

For \textbf{(R1) (R3ii)}, notice $r^*_n=\infty$. There exists functions $\{\theta_k:k\ge1\}$ with compact support on $[0,\infty)$ such that $0 \leq \theta_k \leq 1$ and $\theta_k\uparrow 1$ and $|\theta'_k|<1/k$. Then for any $\cc^*\in\mathfrak C_\pm^{n*}$,
\begin{align*}
\EE_{\pm,1}^{n*}(\cc^*(1-\theta_k),\cc^*(1-\theta_k)) &=\int_0^\infty\cc^{*2}(1-\theta_k)^2 d\mm_\pm^{n*}+ \frac12\int_0^\infty\left( {\cc^*}'(1-\theta_k)+ \cc^*\theta_k'\right)^2 \left(h_\pm^*\right)^2dx\\
&\le \int_0^\infty\cc^{*2}\left[ (1-\theta_k)^2+ {\theta_k'}^2\right] d\mm_\pm^{n*}+ \int_0^\infty\left( {\cc^*}'\right)^2(1-\theta_k)^2 \left(h_\pm^*\right)^2dx\\
&\to 0
\end{align*}
as $k\to\infty$. Hence $\mathfrak C_\pm^{n*}\cap C_c(J_\pm^{n*})$ is $\EE_{\pm,1}^{n*}$-dense in $\mathfrak C_\pm^{n*}$.

The case \textbf{(R3i)} is different for $\mathfrak C_+^{n*}$ and $\mathfrak C_-^{n*}$. For $\cc^*\in\mathfrak C_-^{n*}$, let
\[
\cc^*_k(x)=\cc^*(0)+\int_0^x\left( (-k)\vee{\cc^*}'(z)\right) \wedge k\ dz.
\]
Then $\cc^*_k\in C([0, r^*_n])$ and $\cc^*_k(x)\to\cc^*(x)$ for $x\in[0, r^*_n)$ by the dominated convergence theorem. Note that $|\cc^*_k(x)-\cc^*(x)| \le \int_0^x({\cc^*}'(z)-k)^+ dz$, and
\begin{align*}
\left(\int_0^x |{\cc^*}'(z)|dz\right)^2 &\le \int_0^xh^*_+(z)^2dz\int_0^x {\cc^*}'(z)^2 h^*_-(z)^2 dz\\
&\le M\int_0^xh^*_+(z)^2dz
\end{align*}
for some constant $M$. We have
\begin{align*}
\int_0^{r^*_n}\left(\int_0^x |{\cc^*}'(z)|dz\right)^2 \mm_-^{n*}(dx) &\le M\int_0^{r^*_n} \mm_-^{n*}(dx)\int_0^xh^*_+(z)^2dz\\
&=M\int_{l_n}^\infty h_-(x)^2dx\int_{l_n}^x h_+(z)^21_{W_n}(z)d\tt_n(z)\\
&=M\int_{l_n}^\infty h_+(z)^2 1_{W_n}(z)d\tt_n(z)\int_z^\infty h_-(x)^2dx\\
&=(2\sqrt{2\alpha})^{-1}M\int_{l_n}^\infty h_+(z)^2h_-(x)^2 1_{W_n}(z)d\tt_n(z)\\
&=(2\sqrt{2\alpha})^{-1}M r^*_n<\infty.
\end{align*}
Hence $\cc^*_k\to\cc^*$ in $\EE_{-,1}^{n*}$ again by the dominated convergence theorem. For $\cc^*\in\mathfrak C_+^{n*}$, since
\begin{align*}
|\cc^*(x)-\cc^*(y)| &\le \int_x^y|{c^*}'(z)|dz\\
&\le \int_x^y h^*_-(z)^2dz\int_x^y{c^*}'(z)^2h^*_+(z)^2dz\to 0
\end{align*}
as $x,y\to r^*_n$, it follows that $\cc^*(r^*_n):=\lim_{x\uparrow r^*_n}\cc^*(x)$ exists. Furthermore, if $\cc^*(r^*_n)\ne0$, there exists $\delta,\varepsilon>0$ such that $\cc^*(r^*_n)^2>\varepsilon$ on $(r^*_n-\delta,r^*_n)$, therefore
\[
\int_0^{r^*_n}\cc^*(x)^2\mm_+^{n*}(dx)\ge \int_{r^*_n-\delta}^{r^*_n}\varepsilon\ \mm_+^{n*}(dx)=\infty,
\]
which is a contradiction. It follows that $\cc^*(r^*_n)=0$ and $\mathfrak C_+^{n*}\subset C_0([0,r^*_n))$, where $C_0$ denotes the space of functions vanish outside compact sets. Hence $(\mathfrak C_+^{n*},\EE_+^{n*})$ is regular (Cf. \cite{CF12} 1.3.12).

For \textbf{(R3iii)}, we also have  $\mathfrak C_+^{n*}\subset C_0([0,r^*_n))$, so the regularity follows.

Finally, to see $(J^{n*}_\pm, \mm^{n*}_\pm,  \mathfrak{C}^{n*}_\pm, \EE^{n*}_\pm)$ is a regular representation, define $\Phi:\mathfrak C^n_\pm\to\mathfrak C^{n*}_\pm$ by
\[
(\Phi\cc)(x)=\cc\circ j^{-1}(x),\quad \cc\in\mathfrak C^n_\pm.
\]
From the definition of $\mathfrak C^{n*}_\pm$ and $\EE^{n*}_\pm$, we see that equalities in \eqref{EQ3UUU} holds. Notice that $\mm^{n*}_\pm(J^{n*}_\pm\setminus j_n(I_n))=0$, and $j_n(I_n)\setminus J^{n*}_\pm=\emptyset$ except the case $(\mathfrak{C}^n_+, \EE_+)$ with \textbf{(R3iii)} (resp. $(\mathfrak{C}^n_-, \EE_-)$ with \textbf{(L3iii)}), and in that case $\cc=0$ on $[r_n,\infty)$ (resp. on $(-\infty,l_n]$).
\end{proof}

\begin{appendices}
\section{The orthogonal complements for Dirichlet subspaces of $H^1(\mathbb{R})$}
The conclusions under subspace situation can be derived similarly. Let $(\EE,\FF)$ be a regular Dirichlet subspace of $(\frac12\mathbf{D},H^1(\mathbb{R}))$ and $\mathcal{G}_\alpha$ be its $\alpha$-orthogonal complement. Then
\[
\begin{aligned}
&\mathcal{F}=\left\{f\in L^2(\mathbb{R}):\ f\ll \ss,\ \int_\mathbb{R}\left(\frac{df}{d\ss} \right)^2d\ss<\infty \right\}, \\
&\EE(f,g)=\frac{1}{2}\int_\mathbb{R} \frac{df}{d\ss}\frac{dg}{d\ss}d\ss,\quad f,g\in \FF,
\end{aligned}
\]
where $\ss$ is a strictly increasing and absolutely continuous function on $\mathbb{R}$ satisfying
\[
\ss'(x)=0\text{ or }1,\quad\text{a.e.}
\]
Define $G:=\{x\in\mathbb{R}:\ss'(x)=1 \}$, then (see \cite{LY14} and \cite{LS16})
\[
\mathcal{G}_\alpha=\left\{f\in H^1(\mathbb{R}):f'(y)-f'(x)=2\alpha\int_x^yf(z)dz,\text{ a.e. } x,y\in G \right\}.
\]
The proofs of the following theorems are actually the same as those in the previous sections, so we omit them.

\begin{theorem}
Suppose $f\in H^1(\mathbb{R})$. Then $f\in\mathcal{G}_\alpha$ if and only if
\[
f(x)=\frac{\cc_+(x)}{2}e^{\sqrt{2\alpha}x}+ \frac{\cc_-(x)}{2}e^{-\sqrt{2\alpha}x},\quad x\in\mathbb{R},
\]
where $(\cc_+,\cc_-)$ is a pair of functions on $\mathbb{R}$ satisfying $\cc_\pm\ll dx$, $d\cc_\pm/dx=0$ on $G$, and
\[
\frac{d\cc_+}{dx}e^{\sqrt{2\alpha}x}=\frac{d\cc_-}{dx}e^{-\sqrt{2\alpha}x}\text{ on }\mathbb{R}.
\]
Furthermore, let $\mm_\pm(dx)=e^{\pm 2\sqrt{2\alpha}x}dx$ and assume $f$ has the above form, then $f\in H^1(\mathbb{R})$ is equivalent to $\cc_+\in H^1(\mathbb{R};\mm_+)$ and $\cc_-\in H^1(\mathbb{R};\mm_-)$, where
\[
	H^1(\mathbb{R}; \mm_\pm):=\left\{u\in L^2(\mathbb{R},\mm_\pm): u'\in L^2(\mathbb{R}, \mm_\pm) \right\}.
\]
\end{theorem}

Define
\begin{align*}
&\mathfrak C_\pm:=\{\cc\in H^1(\mathbb{R};\mm_\pm):d\cc/dx=0\text{ on }G \},\\
&\EE_\pm(\cc_1,\cc_2):=\frac{1}{2}\int_\mathbb{R} \frac{d\cc_1}{dx}\frac{d\cc_2}{dx}d\mm_\pm,\quad \cc_1,\cc_2\in\mathfrak C_\pm.
\end{align*}
Further define the mapping
\[
\Gamma_\pm: \mathcal{G}_\alpha \rightarrow \mathfrak C_\pm,\, f\mapsto \cc_\pm.
\]
\begin{theorem}
The mapping $\Gamma_\pm$ is a linear bijection and there exists a constant $K_\alpha>1$ such that for any $f\in\mathcal{G}_\alpha$,
 \[
 K_\alpha^{-1}\EE_{\pm, 1}(\Gamma_\pm f, \Gamma_\pm f)\leq \EE_1(f,f)\leq K_\alpha\EE_{\pm, 1}(\Gamma_\pm f, \Gamma_\pm f).
 \]
\end{theorem}
It follows that $(\EE_\pm,\mathfrak C_\pm)$ is a Dirichlet form on $L^2(\mathbb{R})$ in the wide sense, and we have the same expression as Corollary \ref{C-FC}. So the regular representation can be derived exactly the same way as before with cases \textbf{(L3)} and \textbf{(R3)}.
\end{appendices}


\begin{thebibliography}{}

\bibitem{CF12}
Chen, Z.-Q., Fukushima, M.: Symmetric Markov processes, time change, and boundary theory. Princeton University Press, Princeton, NJ (2012).

\bibitem{FFY05}
Fang, X., Fukushima, M., Ying, J.: On regular Dirichlet subspaces of $H^1(I)$ and associated linear diffusions. Osaka J. Math. 42, 27-41 (2005).

\bibitem{F71}
Fukushima, M.: Regular representations of Dirichlet spaces. Trans. Amer. Math. Soc. 155, 455-473 (1971).

\bibitem{FOT11}
Fukushima, M., Oshima, Y., Takeda, M.: Dirichlet forms and symmetric Markov processes. Walter de Gruyter \& Co., Berlin (2011).

\bibitem{LS16}
Li, L., Song, X.: The $\alpha$-orthogonal complements of regular Dirichlet subspaces for one-dimensional Brownian motion. Sci. China Math. 59(10), 2019-2026 (2016).

\bibitem{LY15}
Li, L., Ying, J.: Regular subspaces of Dirichlet forms. In: Festschrift Masatoshi Fukushima. pp. 397-420. World Sci. Publ., Hackensack, NJ (2015).

\bibitem{LY14}
Li, L., Ying, J.: On structure of regular Dirichlet subspaces for one-dimensional Brownian motion. to appear in Ann. Probab.

\bibitem{LY16}
Li, L., Ying, J.: Regular Dirichlet extensions of one-dimensional Brownian motion. arXiv: 1606.00630. 

\bibitem{LS16-1}
Song, X., Li, L.: Regular Dirichlet subspaces and Mosco convergence. Chin. Ann. Math. Ser. A. 37(1), 1-14 (2016).
\end{thebibliography}
\end{document}